\documentclass[11pt,reqno]{amsart}
\usepackage{a4wide}
\usepackage{amsmath,amsfonts,mathscinet,amssymb,amsthm}
\usepackage{accents}
\usepackage{enumitem}
\usepackage{constants}
\usepackage{tikz-cd}
\usepackage{todonotes}
\setlist[enumerate]{label={\rm(\roman*)}}
\newtheorem{theorem}{Theorem}[section]
\newtheorem{proposition}[theorem]{Proposition}
\newtheorem{lemma}[theorem]{Lemma}
\newtheorem{corollary}[theorem]{Corollary}
\theoremstyle{remark}
\newtheorem{remark}[theorem]{Remark}
\numberwithin{equation}{section}

\expandafter\let\expandafter\oldproof\csname\string\proof\endcsname
\let\oldendproof\endproof
\renewenvironment{proof}[1][\proofname]{%
  \oldproof[\bf #1]%
}{\oldendproof}

\def\M{\mathcal M}
\def\MM{\mathcal M_+}
\def\N{\mathbb N}
\def\Z{\mathbb Z}
\def\R{\mathbb R}
\def\K{\mathbb K}

\newcommand{\ds}{\mathrm{\,d}s}
\newcommand{\dt}{\mathrm{\,d}t}

\newcommand{\dy}{\mathrm{\,d}y}

\newcommand{\jq}{\frac1q}
\newcommand{\jp}{\frac1p}

\newcommand{\mjp}{{-\frac1p}}
\newcommand{\pq}{{\frac pq}}

\newcommand{\qp}{{\frac qp}}

\newcommand{\mqp}{{-\frac qp}}

\newcommand{\sqa}{\mathbf{a}}
\newcommand{\sqb}{\mathbf{b}}
\newcommand{\sqc}{\mathbf{c}}
\newcommand{\sqv}{\mathbf{v}}
\newcommand{\sqw}{\mathbf{w}}
\newcommand{\squ}{\mathbf{u}}

\newcommand{\udol}{\underaccent{\downarrow}u}
\newcommand{\uhor}{\underaccent{\uparrow}u}
\newcommand{\squdol}{\underaccent{\downarrow}\squ}
\newcommand{\squhor}{\underaccent{\uparrow}\squ}

\newcommand{\uhdol}{\accentset{\downarrow}u}
\newcommand{\uhhor}{\accentset{\uparrow}u}
\newcommand{\squhdol}{\accentset{\downarrow}\squ}
\newcommand{\squhhor}{\accentset{\uparrow}\squ}

\newcommand{\RpZ}{\mathbb{R}_+^\mathbb{Z}}

\newcommand{\sumZ}{\sum_{n\in\Z}}

\newcommand{\eps}{\varepsilon}

\begin{document}

\title{Weighted inequalities for discrete iterated Hardy operators}

\author[Amiran Gogatishvili]{Amiran Gogatishvili\textsuperscript{1}}
\author[Martin K\v repela]{Martin K\v repela\textsuperscript{2}}
\author[Rastislav O{\soft{l}}hava]{Rastislav O{\soft{l}}hava\textsuperscript{3,4}}
\author[Lubo\v s Pick]{Lubo\v s Pick\textsuperscript{5}}

\email[A.~Gogatishvili]{gogatish@math.cas.cz}
\urladdr{0000-0003-3459-0355}
\email[M.~K\v repela]{martin.krepela@math.uni-freiburg.de}
\urladdr{0000-0003-0234-1645}
\email[R.~O\soft{l}hava]{olhava@karlin.mff.cuni.cz}
\urladdr{0000-0002-9930-0454}
\email[L.~Pick]{pick@karlin.mff.cuni.cz}
\urladdr{0000-0002-3584-1454}

\address{\textsuperscript{1}%
 Institute of Mathematics,
 Academy of Sciences of the Czech Republic,
 \v Zitn\'a~25,
 115~67 Praha~1,
 Czech Republic}

\address{\textsuperscript{2}%
	Department of Applied Mathematics,
	University of Freiburg, Ernst-Zermelo-Stra\ss e~1,
	791~04 Freiburg, Germany}

\address{\textsuperscript{3}%
 Department of Mathematical Analysis,
	Faculty of Mathematics and Physics,
	Charles University,
	Sokolovsk\'a~83,
	186~75 Praha~8,
	Czech Republic}

\address{\textsuperscript{4}%
 Institute of Applied Mathematics
 and Information Technologies,
 Faculty of Science,
 Charles University,
 Al\-ber\-tov 6,
 128~43 Praha~2,
 Czech Republic}

\address{\textsuperscript{5}%
	Department of Mathematical Analysis,
	Faculty of Mathematics and Physics,
	Charles University,
	Sokolovsk\'a~83,
	186~75 Praha~8,
	Czech Republic}


\thanks{This research was supported by the grants P201-13-14743S and P201-18-00580S of the Czech Science Foundation and by the grant 8X17028 of the Czech Ministry of Education. The research of A. Gogatishvili was partially supported by Shota Rustaveli National Science Foundation (SRNSF), grant no: FR17-589. }

\begin{abstract}
We characterize a~three-weight inequality for an~iterated discrete Hardy-type operator. In the case when the domain space is a~weighted space $\ell^{p}$ with $p\in(0,1]$, we develop characterizations which enable us to reduce the problem to another one with $p=1$. This, in turn, makes it possible to establish an~equivalence of the weighted discrete inequality to an~appropriate inequality for iterated Hardy-type operators acting on measurable functions defined on $\R$, for all cases of involved positive exponents.
\end{abstract}


\maketitle

\section{Introduction}

In this paper we focus on a~three-weight inequality for the composition of a~discrete supremal and integral Hardy operator. Let us denote by $\RpZ$ the space of all double-infinite sequences of positive (nonnegative) real numbers. We are interested in the question under what conditions on given $\squ,\sqv,\sqw\in\RpZ$ there exist constants $\Cl{discrete_supremal}, \Cl{d-antigop}\in(0,\infty)$ such that the inequalities
\begin{equation}\label{E:d-gop}
	\Bigg(\sum_{n\in\Z}\Bigg(\sup_{i\ge n}u_i\sum_{k\le i} a_k\Bigg)^{q} w_n\Bigg)^{\frac 1q}
		\le \Cr{discrete_supremal}
	\Bigg(\sum_{n\in\Z}a_n\sp pv_n\Bigg)\sp{\frac 1p}
\end{equation}
and
	\begin{equation}\label{E:d-antigop}
		\Bigg(\sum_{n\in\Z}\Bigg(\sup_{i\ge n} u_i \sum_{k\ge i} a_k\Bigg)^{q} w_n\Bigg)^{\frac 1q}
		\le \Cr{d-antigop}
		\Bigg(\sumZ a_n^p v_n\Bigg)^\jp
	\end{equation}
hold for every sequence $\sqa\in\RpZ$.
We study several aspects of such an~inequality including its relationship to an~analogous one for integral operators.

Before continuing, let us recall that \eqref{E:d-gop} being satisfied for all $\sqa\in\RpZ$ is equivalent to
  \begin{equation}\label{E:dual-trivka}
	\Bigg(\sum_{n\in\Z}\Bigg(\sup_{i\le n}\overline{u}_i\sum_{k\ge i} a_k\Bigg)^{q} \overline{w}_{n}\Bigg)^{\frac 1q}
		\le \Cr{discrete_supremal}
	\Bigg(\sum_{n\in\Z}a_n\sp p \overline{v}_n\Bigg)\sp{\frac 1p},
  \end{equation}
also being satisfied for all $\sqa\in\RpZ$. This is obvious by the index change $\overline{u}_n = u_{-n}$, $\overline{v}_n = v_{-n}$ and $\overline{w}_n = w_{-n}$. Analogously, the inequality
\begin{equation}\label{E:dual-antigop-trivka}
	\Bigg(\sum_{n\in\Z}\Bigg(\sup_{i\le n}\overline{u}_i\sum_{k\le i} a_k\Bigg)^{q} \overline{w}_{n}\Bigg)^{\frac 1q}
		\le \Cr{discrete_supremal}
	\Bigg(\sum_{n\in\Z}a_n\sp p \overline{v}_n\Bigg)\sp{\frac 1p}
  \end{equation}
is equivalent to~\eqref{E:d-antigop}. It is common to refer to \eqref{E:dual-trivka} and~\eqref{E:dual-antigop-trivka} as to the \emph{dual versions} of \eqref{E:d-gop} and~\eqref{E:d-antigop}, respectively. In contrast, inequalities \eqref{E:d-gop} and \eqref{E:d-antigop} (hence also  \eqref{E:dual-trivka} and~\eqref{E:dual-antigop-trivka}) are essentially different.

The success that the theory of weighted inequalities has seen in last three decades can be credited greatly to a~clever combination of classical techniques such as symmetrization or interpolation with new methods such as discretization (the blocking technique), antidiscretization, reduction theorems, and the use of supremum operators.

The research of problems in mathematical physics often leads to the investigation of certain Sobolev-type embeddings. Under certain circumstances, these can be quite successfully attacked by classical symmetrization techniques. After performing this step, one often faces some kind of an~inequality involving operators acting on monotone functions. Handling monotone functions is, however, in general substantially more difficult than working with general nonnegative functions.

There are several possibilities how to continue at this stage. One of the important ones is the use of the so-called reduction theorems, in which the inequality involving monotone functions is equivalently replaced with an~inequality (or inequalities) involving general nonnegative functions.

For certain types of technically difficult inequalities involving monotone functions, stronger tools have to be used. One of such tools that has proved its merit beyond any doubt, is discretization. Discretization techniques replace weighted inequalities involving integrals with those involving sums. The basic advantage of this step is that discrete inequalities can be effectively manipulated with the help of the so-called blocking technique (see the comprehensive treatment in~\cite{GE98}. The drawback is the fact that verification of the discretized conditions on weight functions in practice is virtually impossible. So here we face the danger of replacing one mystery with another one without making much progress. For this reason, a~substantial effort has been spent in order to develop \textit{antidiscretization} techniques (the pivotal paper in this direction is~\cite{GP-discr}). After performing antidiscretization, one gets manageable and easily verifiable conditions for weighted inequalities that could not be obtained otherwise. Let us note that this approach brought a~significant progress to theory of function spaces and the study of properties of operators on function spaces and several long-standing open problems were solved thanks to it. A particular impact could be seen, for instance, to classical Lorentz spaces or to Orlicz spaces (see, for instance,~\cite{ACS17,len15,GKPS,vej17,vej19,CM19} and more).

One of the most important topics intensively studied in the recent theory of weighted inequalities is that of handling \textit{iterated} operators. The reason stems from the wide field of applications, see for example~\cite{GM1,GM2,Kre1,GKPS,ACS17} and the references therein.

One of the basic problems in the theory of weighted inequalities is the comparison of discrete inequalities to their continuous analogues. Consider, for example, a~classical discrete Hardy-type inequality
\begin{equation}\label{E:discrete_hardy}
	\Bigg(\sum_{n\in\Z}\Bigg(\sum_{i\ge n} a_i \Bigg)^q w_n\Bigg)^{\frac 1q}
		\leq \Cr{discrete_hardy}
	\Bigg(\sum_{n\in\Z}a_n^p v_n \Bigg)^{\frac 1p},
\end{equation}
which is supposed to hold for all $\sqa\in\RpZ$ with the same constant $\Cl{discrete_hardy}$, and where $\sqv,\sqw\in\RpZ$ are fixed sequences (weights). Compare this to its ``continuous'' analogue
\begin{equation}\label{E:integral_hardy}
    \Bigg(\int_0\sp{\infty}\Bigg(\int_t^{\infty}f(s)\ds\Bigg)\sp qw(t)\dt\Bigg)\sp{\frac 1q}
    \leq \Cr{integral_hardy}
    \Bigg(\int_0\sp{\infty}f(t)\sp pv(t)\dt\Bigg)\sp{\frac 1p}
\end{equation}
which is to hold with a~constant $\Cl{integral_hardy}$ for all positive measurable functions $f$ on $\R$. In here, the weights $v$, $w$ are fixed positive measurable functions. The relation between the two inequalities is materialized through setting
\begin{equation*}
    v(t)=\sum_{n\in\Z} v_n\chi_{[n,n+1)}(t),\qquad
    w(t)=\sum_{n\in\Z} w_n\chi_{[n,n+1)}(t)
\end{equation*}
for all $t\in\R$.
While~\eqref{E:discrete_hardy} and~\eqref{E:integral_hardy} are rather easily seen to be equivalent for $p\geq 1$, the situation is dramatically different when $p\in(0,1)$. In that case it is not difficult to realize
that~\eqref{E:integral_hardy} cannot hold for any nontrivial weights, because one can always find a~function $f$ for which the right-hand side of~\eqref{E:integral_hardy} is finite but which is at the same time not locally
integrable, hence turning the left hand side to infinity. On the other hand,~\eqref{E:discrete_hardy} can still be satisfied for a~wide variety of nontrivial weight sequences.
One of our principal goals in this paper is to show that, nevertheless, an~appropriate continuous analogue can be found even for $p\in(0,1)$. To achieve this result, we combine a certain scaling argument with a~powerful technique
based on a~somewhat surprising equivalence of several weighted inequalities. We then employ the fact that the case $p=1$ is a meeting point of the separated worlds. It is worth to illustrate this technique in more detail. The point of departure is a~chain of elementary inequalities, namely
\begin{equation}\label{E:elm}
    \sup_{i\ge n}a_i \le \sum_{i\ge n} a_i\le \Bigg(\sum_{i\ge n} a_i^p\Bigg)^\jp.
\end{equation}
This is obviously true for every $p\in(0,1]$, $n\in\Z$ and $\sqa\in\RpZ$. It immediately follows from~\eqref{E:elm} that if $p\in(0,1]$ and the sequences $\sqv,\sqw$ are such that the inequality
  \begin{equation}\label{E:discrete_hardy_p}
    \Bigg( \sum_{n\in\Z} \Bigg( \sum_{i\ge n} a_i^p \Bigg)^{\frac qp} w_n \Bigg)^\jq
    \leq \Cr{discrete_hardy}
    \Bigg( \sum_{n\in\Z}a_n\sp p v_n \Bigg)^\jp
  \end{equation}
holds for every $\sqa\in\RpZ$, then so does \eqref{E:discrete_hardy}. In turn, \eqref{E:discrete_hardy} implies that
  \begin{equation}\label{E:discrete_hardy_sup}
    \Bigg(\sum_{n\in\Z}\Bigg(\sup_{i\ge n}a_i\Bigg)\sp q w_n\Bigg)\sp{\frac 1q}
    \leq \Cr{discrete_hardy}
    \Bigg(\sum_{n\in\Z} a_n\sp pv_n\Bigg)\sp{\frac 1p}
  \end{equation}
holds for all $\sqa\in\RpZ$ as well. The surprising part of the method is that the implication \eqref{E:discrete_hardy_sup}$\Rightarrow$\eqref{E:discrete_hardy_p} holds as well, therefore the three inequalities are in fact
equivalent. It is important to notice that all this is possible only in the case when $p\in(0,1]$, for $p$ bigger than~$1$ the equivalence fails. The technique just described is not entirely new. Similar ideas were used, albeit in a somewhat hidden form,
in the proof of~\cite[Theorem~3.1]{CGMP2}. An~analogous idea works also for continuous-type problems, again for $p\in(0,1]$ only, as shown in~\cite{gogatishvili2007reduction}. The special role of the case $p=1$ (the ``meeting point'' of intervals of parameters in which things are considerably different) can be also seen for instance in~\cite{Si-St96,Sinn94}.

We will present several characterizations of the inequality~\eqref{E:d-gop}, quite different in nature. In the first theorem we state the equivalence of~\eqref{E:d-gop} to an~appropriate integral inequality for functions on $\R$.

We will denote by $\MM$ the collection of all nonnegative measurable functions on $\R$.

\begin{theorem}\label{T:main-prop1}
	Let $p\in[1,\infty)$ and $q\in(0,\infty)$. Let $\squ,\sqv,\sqw\in\RpZ$. Define
		$$
			u=\sum_{n\in\Z} u_n\chi_{[n,n+1)},\qquad
			v=\sum_{n\in\Z} v_n\chi_{[n,n+1)},\qquad
			w=\sum_{n\in\Z} w_n\chi_{[n,n+1)}.
		$$
	Then~\eqref{E:d-gop} holds for every sequence $\sqa\in\RpZ$ if and only if
		\begin{equation}\label{E:spoj-gop}
			\left(\int_\R \left(\sup_{s\ge t}u(s)\int_{-\infty}^sf(y)\dy\right)\sp qw(t)\dt\right)\sp{\frac 1q}
			\leq \Cr{discrete_supremal}
			\left(\int_\R f(t)\sp pv(t)\dt\right)\sp{\frac 1p}
		\end{equation}
	holds for every $f\in\M_+$.
	
	Similarly, \eqref{E:d-antigop} holds for every sequence $\sqa\in\RpZ$ if and only if
		\begin{equation}\label{E:spoj-antigop}
			\left(\int_\R \left(\sup_{s\ge t}u(s)\int_s^{\infty} f(y)\dy\right)^q w(t)\dt\right)^\jq
			\leq \Cr{d-antigop}
			\left(\int_\R f(t)\sp pv(t)\dt\right)\sp{\frac 1p}
		\end{equation}
	holds for every $f\in\M_+$.
\end{theorem}

In Section 2 below we give the main results concerning characterizations of \eqref{E:d-gop} and \eqref{E:d-antigop}. Section 3 contains some auxiliary results and, above all, the equivalent characterizations for the case $p\in(0,1]$. In the final section we give the remaining proofs of the main results.

\section{Discrete iterated Hardy operators}

This section contains the main results concerning boundedness of iterated Hardy-type operators on weighted sequence spaces.

From now on we are going to use the following notation. Let $\squ\in\RpZ$. For $n\in\Z$ we define
$$
\uhhor_n= \sup_{k\le n} u_k , \quad \uhdol_n = \sup_{k\ge n} u_k.
$$
The sequences $\squhhor$ and $\squhdol$ are called the \emph{increasing} and \emph{decreasing upper envelope of} $\squ$, respectively. Next, define
$$
\uhor_n=\inf_{k\ge n} u_k, \quad \udol_n =\inf_{k\le n} u_k.
$$
The sequences $\squhor$ and $\squdol$ are called the \emph{increasing} and \emph{decreasing lower envelope of} $\squ$, respectively.
If $\sqa,\sqb\in\RpZ$ satisfy $a_n \le b_n$ for all $n\in\Z$, we write
$
\sqa \le \sqb.
$
Furthermore, the notation $A\lesssim B$ means that there exists a~constant $C\in(0,\infty)$ depending only on $p$ and $q$ and such that $A\le CB$. We write $A\approx B$ if $A\lesssim B \lesssim A$.

\begin{theorem} \label{T:main-discrete-p-big}
Let $p,q\in(0,\infty)$ and $\squ,\sqv,\sqw\in\RpZ$. Then the least constant $\Cr{discrete_supremal}$ such that
\eqref{E:d-gop} holds for all $\sqa\in\RpZ$ admits the following estimates.
	\begin{enumerate}
			\item
				If $1 < p\le q$, then
					$$
						\Cr{discrete_supremal} \approx \sup_{n\in\Z} \Bigg( \uhdol_n^q \sum_{i\le n} w_i + \sum_{i\ge n} \uhdol_i^{q} w_i \Bigg)^{\frac 1q}\Bigg(\sum_{k\le n}v_k^{\frac{1}{1-p}}\Bigg)\sp{\frac{p-1}{p}}.
					$$
			\item
				If $p > 1$ and $q<p$, then
					\begin{align*}
						\Cr{discrete_supremal} & \approx \Bigg( \sum_{n\in\Z}  \Bigg(\sum_{i\ge n} \uhdol_i^q w_i\Bigg)^{\frac{q}{p-q}} \uhdol_n^{q} w_n \Bigg(\sum_{k\le n}v_k^{\frac{1}{1-p}}\Bigg)\sp{\frac{(p-1)q}{p-q}} \Bigg)^\frac{p-q}{pq}\\
						& \qquad + \Bigg( \sum_{n\in\Z} \Bigg(\sum_{i\le n} w_i\Bigg)^{\frac{q}{p-q}} w_n \sup_{k\ge n}\uhdol_k^{\frac{pq}{p-q}}\Bigg(\sum_{j\le k} v_j^{\frac{1}{1-p}}\Bigg)^{\frac{(p-1)q}{p-q}} \Bigg)^\frac{p-q}{pq}.
					\end{align*}
			\item
				If $0<p\le 1$ and $p\le q$, then
					$$
						\Cr{discrete_supremal} \approx \sup_{n\in\Z}\Bigg( \uhdol^q_n \sum_{i\le n} w_i + \sum_{k\ge n} \uhdol^q_k w_k \Bigg)^{\frac 1q} \sup_{j\le n} v_j^\mjp.
					$$
			\item
				If $0<q<p\le 1$, then
					$$
						\Cr{discrete_supremal} \approx \Bigg( \sum_{n\in\Z} \Bigg(\sum_{i\ge n } w_i\uhdol_i^q \Bigg)^{\frac{q}{p-q}}	\uhdol_{n}^{q} w_n \sup_{k\le n} v_k^{\frac q{q-p}}
						 + \sum_{n\in\Z} \Bigg(\sum_{i\le n} w_i\Bigg)^{\frac{q}{p-q}} w_n \sup_{k\ge n}\uhdol_{k}^{\frac{pq}{p-q}} v_k^{\frac{q}{q-p}} \Bigg)^{\frac{p-q}{pq}}.
					$$
	\end{enumerate}
\end{theorem}

\begin{theorem} \label{T:main-antigop}
Let $p,q\in(0,\infty)$ and $\squ,\sqv,\sqw\in\RpZ$. Then the least constant $\Cr{d-antigop}$ such that
\eqref{E:d-antigop} holds for all $\sqa\in\RpZ$ admits the following estimates.
	\begin{enumerate}
			\item
				If $1 < p\le q$, then
					$$
						\Cr{d-antigop} \approx \sup_{n\in\Z} \Bigg( \sum_{i\le n} w_i \sup_{i\le j\le n} u_j^q \Bigg)^{\frac 1q} \Bigg(\sum_{k\ge n}v_k^{\frac{1}{1-p}}\Bigg)^{\frac{p-1}{p}}.
					$$
			\item
				If $p > 1$ and $q<p$, then
					\begin{align*}
						\Cr{d-antigop} & \approx \Bigg( \sum_{n\in\Z}  \Bigg(\sum_{i\ge n} w_i\Bigg)^{\frac{q}{p-q}} w_n^{\frac q{p-q}} \sup_{k\ge n} u^{\frac{pq}{p-q}}_k \Bigg(\sum_{m\le k}v_m^{\frac{1}{1-p}}\Bigg)\sp{\frac{(p-1)q}{p-q}} \Bigg)^\frac{p-q}{pq}\\
						& \qquad + \Bigg( \sum_{n\in\Z} \Bigg(\sum_{i\le n} w_i \sup_{i\le j\le n} u^q_j \Bigg)^{\frac{q}{p-q}} w_n \sup_{k\ge n} u_k^q \Bigg(\sum_{m\ge k} v_m^{\frac{1}{1-p}}\Bigg)^{\frac{(p-1)q}{p-q}} \Bigg)^\frac{p-q}{pq}.
					\end{align*}
			\item
				If $0<p\le 1$ and $p\le q$, then
					$$
						\Cr{d-antigop} \approx \sup_{n\in\Z}\Bigg( u^q_n \sum_{i\le n} w_i + \sum_{k\ge n} u^q_k w_k \Bigg)^{\frac 1q} \sup_{j\le n} v_j^\mjp.
					$$
			\item
				If $0<q<p\le 1$, then
					\begin{align*}
						\Cr{d-antigop} & \approx \Bigg( \sum_{n\in\Z} \Bigg( \sum_{i\le n} w_i \Bigg)^{\frac{q}{p-q}} w_n \sup_{k\ge n} u_k^{\frac q{p-q}} \sup_{m\ge k} v_m^{\frac q{q-p}}  \Bigg)^\frac{p-q}{pq} \\
						& \qquad  + \Bigg( \sum_{n\in\Z} \Bigg(\sum_{i\le n} w_i \sup_{i\le j\le n} u^q_j \Bigg)^{\frac{q}{p-q}} w_n \sup_{k\ge n} u^q_k  \sup_{m\ge k} v_m^{\frac{q}{q-p}} \Bigg)^{\frac{p-q}{pq}} .
					\end{align*}
	\end{enumerate}
\end{theorem}

\section{Equivalence theorems for $p\in(0,1]$}
	
In this section, after presenting some auxiliary results, we show an~equivalence principle for supremal and integral Hardy operators in the case $p\in(0,1]$. These results establish a~link between discrete and continuous Hardy-type inequalities for such $p$, but they are of independent interest.

The first preliminary result is an~extension of \cite[Theorem 3.1]{Sinn03} concerning ``transferring monotonicity'' to the weight sequence on the right-hand side. In here, we use the following notation, for $\sqa\in\RpZ$,
$$
S\sqa_n = \sup_{j\le n} a_j, \qquad I\sqa_j= \sum_{j\le n} a_j.
$$
Hence $S\sqa, I\sqa\in\overline{\R}_+^\Z$ and $S\sqa_n$, $I\sqa_n$ are the $n$-th entries of $S\sqa$ and $I\sqa$, respectively.

\begin{lemma}\label{L:Sinn-trik}
	Let $\squ\in\RpZ$. Let $\varphi:\RpZ\to\R_+$ be a~functional such that there exists a~sequence $\sqc\in\RpZ$ with a~finite number of non-zero entries for which $\varphi(\sqc)>0$. In addition to this, assume that  $\varphi$ satisfies
	\begin{equation}\label{E:SiS}
	S\sqa\le S\sqb \quad \Longrightarrow \quad \varphi(\sqa) \le \varphi(\sqb)
	\end{equation}
	or
	\begin{equation}\label{E:SiI}
	I\sqa\le I\sqb \quad \Longrightarrow \quad \varphi(\sqa) \le \varphi(\sqb)
	\end{equation}
	for all $\sqa,\sqb\in\RpZ$. Then we have
	\begin{equation}\label{E:zmono}
	\sup_{\sqa\in\RpZ} \frac{ \varphi(\sqa) }{ \displaystyle\sumZ a_n u_n } = \sup_{\sqa\in\RpZ} \frac{ \varphi(\sqa) }{ \displaystyle\sumZ a_n \udol_n }.
	\end{equation}
\end{lemma}

\begin{proof}
	The assertion involving an~operator satisfying \eqref{E:SiI} follows from the proof of \cite[Theorem 3.1]{Sinn03}. The proof for the case \eqref{E:SiS} is rather similar but we give it here for the sake of completeness.
	
	The inequality ``$\le$'' is obvious since $\squdol\le \squ$. We have to show ``$\ge$''. First assume that $\squdol$ is identically zero. By the properties of $\varphi$, there exists a~finite set of indices $M\subset \Z$ and a~sequence $\sqc\in\R_+^\Z$ such that $\varphi(\sqc)>0$ and $c_n=0$ unless $n\in M$. Let $\eps>0$. Since $\liminf_{n\to-\infty} u_n = 0$, there exists $N\in\Z$ such that $N\le \min M$ and $u_N<\eps$. Define
	$
	b_N=\max_{n\in M} c_n
	$
	and $b_n= 0$ for all $n\in\Z\setminus \{N\}$. The sequence $\sqb=\{b_n\}_{n\in\Z}$ satisfies $S\sqb\ge S\sqc$, thus also $\varphi(\sqb)\ge \varphi(\sqc)$. Moreover, we have
	$$
	\sum_{n\in\Z} b_nu_n = b_Nu_N < \eps\,\max_{n\in M} c_n.
	$$
	Hence,
	$$
	\frac{ \varphi(\sqb) }{ \sum_{n\in\Z} b_n u_n } >  \frac{ \varphi(\sqb) }{ \eps\,\max_{n\in M} c_n } \ge \frac{ \varphi(\sqc) }{ \eps\,\max_{n\in M} c_n },
	$$
	and therefore
	$$
	\sup_{\sqa\in\R_+^\N} \frac{ \varphi(\sqa) }{ \sum_{n\in\Z} a_n u_n } > \frac{ \varphi(\sqc) }{ \eps\,\max_{n\in M} c_n }.
	$$
	Since $\eps>0$ was arbitrary, we have
	$$
	\sup_{\sqa\in\R_+^\N} \frac{ \varphi(\sqa) }{ \sum_{n\in\Z} a_n u_n } = \infty,
	$$
	so the inequality ``$\ge$'' in \eqref{E:zmono} is obviously satisfied.
	
	In the following, we assume that $\squdol$ is not identically zero, hence
	\begin{equation}\label{E:nenula}
	\lim_{n\to-\infty} \udol_n = \liminf_{n\to-\infty} u_n > 0.
	\end{equation}
	Let $\eps>0$. By definition of the envelope and \eqref{E:nenula}, there exists an~index $n_0\in\Z$ such that
	$$
	u_{n_0} \le (1+\eps) \udol_{n_0}.
	$$
	Now we define a~sequence $\{n_k\}$ recursively.	At first, we construct the ``positive part'' with indices $k>0$ as follows. If $k\in\N$, $n_{k-1}$ is defined and $n_{k-1}<\infty$, define
	$$
	n_k= \inf\left\{j\in\Z\,\big|\, j>n_{k-1},\, u_j \le (1+\eps) \udol_j \right\},
	$$
	where $\inf \emptyset =\infty$. In this way, we get a~strictly increasing sequence of indices $\{n_k\}_{n=0}^{K}$ which is either finite with $K\in\N$ and $n_{K}=\infty$, or infinite with $K=\infty$.
	Furthermore, we construct the ``negative part'' with indices $k<0$. If $k\in\Z$, $k<0$, is such that $n_{k+1}$ is already defined, put
	$$
	n_k= \sup\left\{j\in\Z\,\big|\, j<n_{k+1},\, u_j \le (1+\eps) \udol_j \right\}.
	$$
	In this case, the set over which the supremum is taken is nonempty, by the definition of $\squdol$ and \eqref{E:nenula}. Hence, altogether we obtain a~strictly increasing sequence of indices $\{n_k\}_{n=-\infty}^{K}$ such that
	\begin{equation}\label{E:blizko}
	u_{n_k} \le (1+\eps) \udol_{n_k}
	\end{equation}
	and
	\begin{equation}\label{E:kst}
	\udol_n = \udol_{n_k} \quad \text{for all } n\in\{n_k,\ldots,n_{k+1}-1\}
	\end{equation}
	for all $k\in\Z$ such that $k< K$. To verify \eqref{E:kst}, suppose that if $\udol_{n_k}>\udol_{j}$ for some $j\in\Z$, $j>n_k$. Without loss of generality, $j$ is the smallest index with this property. Then necessarily $\udol_j=u_j$ by definition of the envelope, and thus $n_{k+1}\le j$ by definition of $\{n_k\}$.
	
	Let us note that if $\squdol$ contains no infinite constant subsequence, the above construction may be performed with $\eps=0$  ($K=\infty$ is then guaranteed).
	
	Fix $\sqa\in\RpZ$ arbitrary.
	We define a~sequence $\sqb$ by setting
	\begin{equation}\label{E:defb}
	b_n= \begin{cases}
	\displaystyle\sup_{n_k\le j < n_{k+1}} a_n & \text{ if } n=n_k \text{ for some } k\in\Z,\ k<K,\vspace{2pt}\\
	0 & \text{ else.}
	\end{cases}
	\end{equation}
	It follows that $S\sqa\le S\sqb$. Indeed, for each $n\in\Z$ there exists $k\in\Z$, $k<K$, such that $n_k\le n < n_{k+1}$ and we have, for each $n$,
	$$
	S\sqa_n = \sup_{j\le n} a_j \le \max\left\{ \sup_{j<n_k} a_j,\ \sup_{n_k\le j <n_{k+1}} a_j \right\} = S\sqb_n.
	$$
	Moreover, by \eqref{E:defb}, \eqref{E:kst} and \eqref{E:blizko} one has
	\begin{equation}\label{E:a-vs-b}	
	\sumZ b_n u_n  =  \sum_{k\leq K-1} u_{n_k} \sup_{n_k\le j < n_{k+1}} a_j
	\le (1+\eps)  \sum_{k \leq K-1}  \sup_{n_k\le j < n_{k+1}}  a_j \udol_j
	\le (1+\eps) \sumZ a_n \udol_n.
	\end{equation}
	By the properties of $\varphi$, $S\sqa \le S \sqb$ implies $\varphi(\sqa)\le\varphi(\sqb)$. From this and \eqref{E:a-vs-b} we obtain
	$$
	\frac{ \varphi(\sqa) }{ \displaystyle \sumZ a_n \udol_n }
	\le \frac{ (1+\eps) \varphi(\sqa) }{ \displaystyle \sumZ b_n u_n} \le \frac{ (1+\eps) \varphi(\sqb) }{ \displaystyle \sumZ b_n u_n} \le (1+\eps) \sup_{\sqb\in\RpZ} \frac{  \varphi(\sqb) }{ \displaystyle \sumZ b_n u_n }.
	$$
Since $\sqa\in\RpZ$ and $\eps>0$ were arbitrary, we get the desired inequality
	$$
	\sup \frac{ \varphi(\sqa) }{ \displaystyle \sumZ a_n \udol_n } \le \sup_{\sqb\in\RpZ} \frac{ \varphi(\sqb) }{ \displaystyle \sumZ b_n u_n }.
	$$
\end{proof}

\begin{remark}\label{R:dua1}
	For $\sqa\in\RpZ$, define
	$$
	S^*\sqa_n = \sup_{j\ge n} a_j \quad\text{and} \quad I^*\sqa_j= \sum_{j\ge n} a_j.
	$$
	Lemma \ref{L:Sinn-trik} holds unchanged if we replace $S$ by $S^*$ in \eqref{E:SiS} as well as $I$ by $I^*$ in \eqref{E:SiI}, and $\udol$ by $\uhor$ in \eqref{E:zmono}. To check this, it suffices to perform the index change $\overline{a}_n= a_{-n}$, $n\in\Z$.
\end{remark}

In what follows we are going to use a~blocking technique (see \cite{GE98}). To this end, we need the following definition.
Let $\sqw\in\RpZ$ and $n_0\in\Z$. The \emph{block partition with respect to $\sqw$ starting at $n_0$} is the sequence $\{n_k\}_{k=0}^K$ defined recursively by
\begin{align*}
n_1 & = n_0+1,  \\
n_k & =  \inf \left \{ j\in\Z \ \Big| \ j>n_{k-1},\ \sum_{i\ge j} w_i \ge 2 \sum_{i=n_{k-1}}^{j-1} w_i \right\}\quad\text{for $k\geq 2$}.
\end{align*}
In here, $K\in\Z$ if $\sumZ w_n <\infty$, and $K=\infty$ otherwise. Notice also the convention $\inf \emptyset=\infty$ being used. Furthermore, define
$$
\K = \left\{ k\in\{ 1, \ldots, K-1\} \, \big| \, n_k < n_{k+1}-1 \right\}.
$$
By the construction, for all $k\in\K$ it holds that
$$
\sum_{i=n_k}^{n_{k+1}-2} w_i < 2 \sum_{i=n_{k-1}}^{n_k-1} w_i.
$$
The reverse inequality holds for all $k\in\{ 1, \ldots, K-2\}$. Here, as well as any other parts of the article where block partitions are used, we assume, without loss of generality, that $K\ge 3$.

The blocking technique relies on the following well-known proposition (see~\cite{GE98,Kre3,GP-discr}).

\begin{proposition}\label{P:dya}
	Let $0<\alpha<\infty$. Then there exists a~constant $C\in(0,\infty)$ such that for any $k_{\min},k_{\max}\in\Z\cup\{\pm\infty\}$, $k_{\min}\le k_{\max}-2$, and any $\sqb,\sqc\in\RpZ$ satisfying $b_{k+1}\ge 2 b_k$  for all $k\in\Z,\ k_{\min}\le k \le k_{\max}-2$, one has
	\[
	\sum_{k=k_{\min}}^{k_{\max}} \left( \sum_{m=k}^{k_{\max}} c_m \right)^\alpha b_k  \le C \sum_{k=k_{\min}}^{k_{\max}} c_k^\alpha b_k,
	\]
	\[
	\sum_{k=k_{\min}}^{k_{\max}} \sup_{k\le m \le k_{\max}} \!\! c_m b_k  \le C  \sum_{k=k_{\min}}^{k_{\max}} c_k b_k .
	\]
	The constant $C$ depends only on $\alpha$.
\end{proposition}

Notice that in the above proposition we have assumed that the index set $\{ k_{\min},\ldots,k_{\max} \}$ contains at least three elements, and the condition $b_{k+1}\ge 2 b_k$ does not need to hold for $k=k_{\max}-1$. As the reader may check very easily, this does not change the validity of the proposition.\\

As the least (optimal) constants are expressed as suprema in the results below, the convention $0^{-\alpha} = \infty$, $\infty^{-\alpha} = 0$ $(\alpha>0)$, $0\cdot\infty = 0$ is in charge.

The first result obtained by the blocking technique involves a~simple Hardy inequality. It may be recovered by examining the characterizations in \cite[Theorem 7.7]{GE98}. Here we present a~direct proof since we are going to use its elements further on.

\begin{lemma}\label{L:prop24}
	Let	$p\in(0,1]$, $q\in(0,\infty)$ and $\sqv,\sqw\in\RpZ$. Define
	\begin{align}
	A_\eqref{10}  & = \sup_{\sqa\in\RpZ} \Bigg( \sumZ w_n \sup_{j\ge n} a^q_j\Bigg)^\jq \Bigg( \sumZ a_n^p v_n \Bigg)^\mjp \label{10},\\
	A_\eqref{11}  & = \sup_{\sqa\in\RpZ} \Bigg( \sumZ w_n \Big[ \sum_{j\ge n} a_j \Big]^q \Bigg)^\jq  \Bigg( \sumZ a_n^p v_n \Bigg)^\mjp \label{11},\\
	A_\eqref{12}  & = \sup_{\sqa\in\RpZ} \Bigg( \sumZ w_n \Big[ \sum_{j\ge n} a_j^p \Big]^\qp  \Bigg)^\jq  \Bigg( \sumZ a_n^p v_n \Bigg)^\mjp \label{12}.
	\end{align}
Then the quantities $A_\eqref{10}$, $A_\eqref{11}$ and $A_\eqref{12}$ are equivalent, and, moreover, the equivalence constants depend only on $p$ and $q$.
\end{lemma}

\begin{proof}
	Since $p\in(0,1]$, the inequalities $A_\eqref{10} \le A_\eqref{11} \le A_\eqref{12}$ follow from~\eqref{E:elm}. We will prove $A_\eqref{12} \le C A_\eqref{10}$ with an~appropriate constant $C$.
	
	By Remark \ref{R:dua1}, we may assume that $\sqv$ is increasing. Let $\sqa\in\RpZ$ be such that $\sumZ a_n^pv_n \in(0,\infty)$. Fix an~arbitrary $n_0\in\Z$. Let $\{n_k\}_{k=0}^K$ be the block partition with respect to $\sqw$ starting at $n_0$.
		We have
	\begin{align*}
	\sum_{n\ge n_0} w_n \Big[ \sum_{j\ge n} a_j^p \Big]^\qp  & = \sum_{k=0}^{K-1} \sum_{n=n_k}^{n_{k+1}-1} w_n \Big[ \sum_{j\ge n} a_j^p \Big]^\qp\\
	& = \sum_{k=0}^{K-2} w_{n_{k+1}-1} \Big[ \sum_{j \ge n_{k+1}-1} a_j^p \Big]^\qp + \sum_{k\in\K} \sum_{n=n_k}^{n_{k+1}-2} w_n \Big[ \sum_{j\ge n} a_j^p \Big]^\qp   \\
	& \lesssim \sum_{k=0}^{K-2} w_{n_{k+1}-1} \Big[ \sum_{j \ge n_{k+1}-1} a_j^p \Big]^\qp + \sum_{k\in\K} \sum_{n=n_{k-1}}^{n_{k}-1} w_n \Big[ \sum_{j\ge n_k-1} a_j^p \Big]^\qp \\
	& \lesssim \sum_{k=0}^{K-2} \sum_{n=n_k}^{n_{k+1}-1} w_n \Big[ \sum_{j \ge n_{k+1}-1} a_j^p \Big]^\qp \\
	& \lesssim \sum_{k=0}^{K-2} \sum_{n=n_k}^{n_{k+1}-1} w_n \Big[ \sum_{j= n_{k+1}-1}^{ n_{k+2}-2} a_j^p \Big]^\qp. 
	\end{align*}
	Here we used the properties of the block partition on the third line, and Proposition \ref{P:dya} on the fifth. Now define the sequence $\sqb\in\RpZ$ by
	$$
	b_n= \begin{cases}
	\displaystyle\Big[\sum_{j= n_{k}-1}^{ n_{k+1}-2} a_j^p \Big]^\jp& \text{if } n=n_{k}-1  \text{ for some }k\in\{1,\ldots,K-1\}, \vspace{2pt}\\
	0 & \text{otherwise}.
	\end{cases}
	$$
	Since $\sqv$ is increasing, we have
	$$
	\sumZ b_n^p v_n = \sum_{k=1}^{K-1} \sum_{j= n_{k}-1}^{ n_{k+1}-2} a_j^p v_j \le \sumZ a_n^p v_n.
	$$
	Altogether, we obtain the following chain of relations in which $C\in(0,\infty)$ depends only on $p$ and $q$,
	\begin{align*}
	\Bigg( \sum_{n\ge n_0} w_n \Big[ \sum_{j\ge n} a_j^p \Big]^\qp \Bigg)^\jq \Bigg( \sumZ a^p_n v_n \Bigg)^\mjp
	& \le C \Bigg( \sum_{k=0}^{K-2} \sum_{n=n_k}^{n_{k+1}-1} w_n \Big[ \sum_{j= n_{k+1}-1}^{ n_{k+2}-2} a_j^p \Big]^\qp \Bigg)^\jq \Bigg( \sumZ a^p_n v_n \Bigg)^\mjp \\
	& =  C  \Bigg( \sum_{k=0}^{K-2} \sum_{n=n_k}^{n_{k+1}-1} w_n b_{n_{k+1}-1}^q \Bigg)^\jq \Bigg( \sumZ a^p_n v_n \Bigg)^\mjp \\
	& \le C \Bigg( \sum_{k=0}^{K-2} \sum_{n=n_k}^{n_{k+1}-1} w_n \sup_{j\ge n} b_j^q \Bigg)^\jq \Bigg( \sumZ b_n^p v_n \Bigg)^\mjp \\
	& \le C \sup_{\sqb\in\RpZ} \Bigg( \sumZ w_n \sup_{j\ge n} b_j^q \Bigg)^\jq \Bigg( \sumZ b_n^p v_n \Bigg)^\mjp.
	\end{align*}
	Since $n_0$ was arbitrary, we have
	$$
	\Bigg( \sumZ w_n \Big[ \sum_{j\ge n} a_j^p \Big]^\qp \Bigg)^\jq \Bigg( \sumZ a^p_n v_n \Bigg)^\mjp \le C \sup_{\sqb\in\RpZ} \Bigg( \sumZ w_n \sup_{j\ge n} b_j^q \Bigg)^\jq \Bigg( \sumZ b_n^p v_n \Bigg)^\mjp
	$$
	with the same $C$. If $\sumZ a_n^p v_n = 0$, the inequality holds trivially. If $\sumZ a_n^p v_n = \infty$, both sides of the inequality are either zero (when $\sqw$ is constant zero) or infinite. Hence, we may take the supremum over $\sqa\in\RpZ$ on the left-hand side, which yields $A_\eqref{12}\le C A_\eqref{10}$.
\end{proof}

An analogous statement to the preceding lemma in the case when $q=\infty$ holds, too. It can be easily proved by interchanging the suprema.

\begin{lemma}\label{L:inft}
	Let	$p\in(0,1]$ and $\sqv,\sqw\in\RpZ$. Then
	$$
	\sup_{\sqa\in\RpZ} \sup_{n\in\Z} u_n \sum_{j\ge n} a_j  \Bigg( \sum_{i\in \Z} a^p_i v_i \Bigg)^\mjp
	= \sup_{\sqa\in\RpZ} \sup_{n\in\Z} u_n \sup_{j\ge n} a_j  \Bigg( \sum_{i\in \Z} a^p_i v_i \Bigg)^\mjp
	= \sup_{n\in\Z} u_n \sup_{j\ge n} v_j^\mjp.
	$$
\end{lemma}

\begin{remark}\label{R:dua2}
	As usual, both Lemmas \ref{L:prop24} and \ref{L:inft} have their ``dual versions'', in which the suprema or sums over $j\ge n$ are replaced by their respective counterparts over $j\le n$. We omit the details.
\end{remark}

We are now in a position to prove a~similar equivalence for the more complicated iterated Hardy operators.

\begin{theorem}\label{T:ekv-antigop}
	Let	$p\in(0,1]$, $q\in(0,\infty)$ and $\sqv,\sqw\in\RpZ$. Define
	\begin{align}
	A_\eqref{ag1} &= \sup_{\sqa\in\RpZ} \Bigg( \sumZ w_n \Big[ \sup_{j\ge n}\,u_j   \sup_{i\ge j} a_i \Big]^q \Bigg)^\jq \Bigg( \sumZ v_n a_n^p \Bigg)^\mjp, \label{ag1}\\
	A_\eqref{ag2} &= \sup_{\sqa\in\RpZ} \Bigg( \sumZ w_n \Big[ \sup_{j\ge n} u_j \sum_{i\ge j} a_i \Big]^q \Bigg)^\jq \Bigg( \sumZ v_n a_n^p \Bigg)^\mjp, \label{ag2}\\
	A_\eqref{ag3} &= \sup_{\sqa\in\RpZ} \Bigg( \sumZ w_n \Big[ \sup_{j\ge n} u_j^p \sum_{i\ge j} a_i^p \Big]^\qp \Bigg)^\jq \Bigg( \sumZ v_n a_n^p \Bigg)^\mjp. \label{ag3}
	\end{align}
	Then $A_\eqref{ag1}$, $A_\eqref{ag2}$ and $A_\eqref{ag3}$ are mutually equivalent, and, moreover, the equivalence constants depend only on $p$ and $q$.	
\end{theorem}

\begin{proof}
Due to~\eqref{E:elm}, only $A_\eqref{ag3} \le C A_\eqref{ag1}$ needs proving. Let $n_0\in\Z$ and let $\{n_k\}_{k=0}^K$ be the block partition with respect to $\sqw$ starting at $n_0$. Without loss of generality we may assume that $K\ge 3$. Let $\sqa\in\RpZ$ be such that $\sumZ a^p_n v_n \in (0,\infty)$.
	Analogously as in Lemma \ref{L:prop24} we have
	\begin{align*}
	\sum_{n\ge n_0} w_n \sup_{j\ge n}\,u_j^q \Big[ \sum_{i \ge j} a_i^p \Big]^\qp
	& = \sum_{k=0}^{K-1} \sum_{n=n_k}^{n_{k+1}-1} w_n \sup_{j\ge n}\,u_j^q \Big[ \sum_{i \ge j} a_i^p \Big]^\qp\\
	& =  \sum_{k=0}^{K-2} w_{n_{k+1}-1} \sup_{j \ge n_{k+1}-1} u_j^q \Big[ \sum_{i \ge j} a_i^p \Big]^\qp \\
	& \qquad + \sum_{k\in\K} \sum_{n=n_k}^{n_{k+1}-2} w_n \sup_{n\le j \le n_{k+1}-2} u_j^q \Big[ \sum_{i\ge j} a_i^p \Big]^\qp  \\
	& \lesssim \sum_{k=0}^{K-2} \sum_{n=n_k}^{n_{k+1}-1} w_n  \sup_{j \ge n_{k+1}-1} u_j^q \Big[\sum_{i \ge j} a_i^p \Big]^\qp\\
	& \lesssim \sum_{k=0}^{K-2} \sum_{n=n_k}^{n_{k+1}-1} w_n  \sup_{n_{k+1}-1 \le j \le n_{k+2}-2} u_j^q \Big[\sum_{i \ge j} a_i^p \Big]^\qp\\
	& \lesssim \sum_{k=0}^{K-2} \sum_{n=n_k}^{n_{k+1}-1} w_n  \sup_{n_{k+1}-1 \le j \le n_{k+2}-2} u_j^q \Big[\sum_{i = j}^{n_{k+2}-2} a_i^p \Big]^\qp\\
	& \qquad	+ \sum_{k=0}^{K-3} \sum_{n=n_k}^{n_{k+1}-1} w_n \sup_{n_{k+1}-1 \le j \le n_{k+2}-2} u_j^q \Big[ \sum_{i \ge n_{k+2}-1} a_i^p \Big]^\qp	\\
	& = B_1 + B_2.
	\end{align*}
	If $k\in\{0,\ldots,K-2\}$ and
	\begin{equation}\label{E:podm}
	\sum_{ n = n_{k+1}-1}^{n_{k+2}-2} a_n^p v_n > 0,
	\end{equation}
	find $c_{n_{k+1}-1}, \ldots, c_{n_{k+2}-2}\ge 0$ such that
	$$
	\sum_{ n = n_{k+1}-1}^{n_{k+2}-2} c_n^p v_n = \sum_{ n = n_{k+1}-1}^{n_{k+2}-2} a_n^p v_n
	$$
	and
	\begin{multline*}
	\sup_{\sqb\in\RpZ} \sup_{n_{k+1}-1 \le j \le n_{k+2}-2} u_j \sup_{j\le i \le n_{k+2}-2} b_i \Bigg( \sum_{ m = n_{k+1}-1}^{n_{k+2}-2} b_m^p v_m \Bigg)^\mjp \\
	\le 2 \sup_{n_{k+1}-1 \le j \le n_{k+2}-2} u_j \sup_{j\le i \le n_{k+2}-2} c_i \Bigg( \sum_{ n = n_{k+1}-1}^{n_{k+2}-2} c_n^p v_n \Bigg)^\mjp.
	\end{multline*}
	For all other indices $n\in\Z$ such that $n\notin\{{n_{k+1}-1}, \ldots, {n_{k+2}-2}\}$ and all $k\in\{0,\ldots,K-2\}$ satisfying \eqref{E:podm} we define $c_n = 0$. In this way we obtain a~sequence $\sqc\in\RpZ$ which moreover satisfies
	$$
	\sumZ c_n^p v_n \le \sum_{n\ge n_0} a_n^p v_n.
	$$
	Using Lemma \ref{L:inft} we get
	\begin{align*}
	B_1 & \le \sum_{k=0}^{K-2} \sum_{n=n_k}^{n_{k+1}-1} w_n \Bigg[ \sup_{\sqb\in\RpZ} \sup_{n_{k+1}-1 \le j \le n_{k+2}-2} u_j \sum_{i = j}^{n_{k+2}-2} b_i \Bigg( \sum_{ m = n_{k+1}-1}^{n_{k+2}-2} b_m^p v_m \Bigg)^\mjp \Bigg]^q \Bigg( \sum_{ n = n_{k+1}-1}^{n_{k+2}-2} a_n^p v_n \Bigg)^\qp\\
	& \lesssim \sum_{k=0}^{K-2} \sum_{n=n_k}^{n_{k+1}-1} w_n \Bigg[ \sup_{\sqb\in\RpZ} \sup_{n_{k+1}-1 \le j \le n_{k+2}-2} u_j \sup_{j\le i \le n_{k+2}-2} b_i \Bigg( \sum_{ m = n_{k+1}-1}^{n_{k+2}-2} b_m^p v_m \Bigg)^\mjp \Bigg]^q \Bigg( \sum_{ n = n_{k+1}-1}^{n_{k+2}-2} a_n^p v_n \Bigg)^\qp\\
	& \lesssim \sum_{k=0}^{K-2} \sum_{n=n_k}^{n_{k+1}-1} w_n \Big[ \sup_{n_{k+1}-1 \le j \le n_{k+2}-2} u_j \sup_{j\le i \le n_{k+2}-2} c_i \Big]^q\\
	& \lesssim \sum_{n\ge n_0} w_n \Big[ \sup_{j \ge n} u_j \sup_{i\ge j} c_i \Big]^q.
	\end{align*}
	Hence,
	\begin{align*}
	B_1 \Bigg( \sum_{n\ge n_0} a_n^p v_n \Bigg)^\mqp  & \lesssim \sum_{n\ge n_0} w_n \Big[ \sup_{j \ge n} u_j \sup_{i\ge j} c_i \Big]^q \Bigg( \sumZ a_n^p v_n \Bigg)^\mqp \\
	& \lesssim \sum_{n\ge n_0} w_n \Big[ \sup_{j \ge n} u_j \sup_{i\ge j} c_i \Big]^q \Bigg( \sumZ c_n^p v_n \Bigg)^\mqp \\
	& \le A_\eqref{ag1}^q.
	\end{align*}
	Lemma \ref{L:prop24} further yields
	\begin{align*}
	B_2 \Bigg( \sum_{n\ge n_0} a_n^p v_n \Bigg)^\mqp
	& \lesssim \sum_{k=0}^{K-3} \sum_{n=n_k}^{n_{k+1}-1} w_n \sup_{n_{k+1}-1 \le j \le n_{k+2}-2} u_j^q \Big[ \sup_{i \ge n_{k+2}-1} a_i \Big]^q	\Bigg( \sumZ a_n^p v_n \Bigg)^\mqp \\
	& \le \sumZ  w_n \Big[ \sup_{j\ge n} u_j  \sup_{i \ge j} a_i \Big]^q	\Bigg( \sumZ a_n^p v_n \Bigg)^\mqp \\
	& \le A_\eqref{ag1}^q.
	\end{align*}
	Finally, we get
	\begin{align*}
	& \Bigg( \sum_{n\ge n_0} w_n \Big[ \sup_{j\ge n}\,u_j \sum_{i \le j} a_i \Big]^q \Bigg)^\jq \Bigg( \sumZ a_n^p v_n \Bigg)^\mjp \\
	& \le \Bigg( \sum_{n\ge n_0} w_n \Big[ \sup_{j\ge n}\,u_j \sum_{i \le j} a_i \Big]^q \Bigg)^\jq \Bigg( \sum_{n\ge n_0} a_n^p v_n \Bigg)^\mjp \\
	& \le C A_\eqref{ag1}.
	\end{align*}
	Since $n_0$ may be arbitrarily small, we obtain, with the same constant $C$, the desired inequality $A_\eqref{ag2} \le C A_\eqref{ag1}$. The cases when $\sumZ a^p_n v_n$ is either zero or infinite can be treated as in the end of the proof of Lemma \ref{L:prop24}.
\end{proof}

\begin{theorem}\label{T:ekv-gop}
	Let	$p\in(0,1]$, $q\in(0,\infty)$ and $\sqv,\sqw\in\RpZ$. Define
	\begin{align}
	A_\eqref{g1} &= \sup_{a\in\RpZ} \Bigg( \sumZ w_n \Big[ \sup_{j\ge n}\,u_j   \sup_{i\le j} a_i \Big]^q \Bigg)^\jq \Bigg( \sumZ v_n a_n^p \Bigg)^\mjp, \label{g1}\\
	A_\eqref{g2} &= \sup_{a\in\RpZ} \Bigg( \sumZ w_n \Big[ \sup_{j\ge n} u_j \sum_{i\le j} a_i \Big]^q \Bigg)^\jq \Bigg( \sumZ v_n a_n^p \Bigg)^\mjp, \label{g2}\\
	A_\eqref{g3} &= \sup_{a\in\RpZ} \Bigg( \sumZ w_n \Big[ \sup_{j\ge n} u_j^p \sum_{i\le j} a_i^p \Big]^\qp \Bigg)^\jq \Bigg( \sumZ v_n a_n^p \Bigg)^\mjp. \label{g3}
	\end{align}
	Then $A_\eqref{g1}$, $A_\eqref{g2}$ and $A_\eqref{g3}$  are equivalent, and, moreover, the equivalence constants depend only on $q$.	
\end{theorem}	

\begin{proof}	
	The proof is essentially the same as that of Theorem~\ref{T:ekv-antigop}. The only minor difference is that, with $\{n_k\}_{k=0}^K$ being the block partition with respect to $\sqw$ starting at $n_0$ and $\sqa\in\RpZ$ being a~sequence such that $\sumZ a^p_n v_n \in (0,\infty)$, we get the following estimate:
	\begin{align*}
	\sum_{n\ge n_0} w_n \Big[ \sup_{j\ge n}\,u_j \sum_{i \le j} a_i \Big]^q
	& \lesssim \sum_{k=0}^{K-2} \sum_{n=n_k}^{n_{k+1}-1} w_n \Big[ \sup_{n_{k+1}-1 \le j \le n_{k+2}-2} u_j \sum_{i=n_{k+1}-2}^j a_i \Big]^q \\
	& \qquad	+ \sum_{k=1}^{K-2} \sum_{n=n_k}^{n_{k+1}-1} w_n \sup_{n_{k+1}-1 \le j \le n_{k+2}-2} u_j^q \Big[ \sum_{i \le n_{k+1}-2} a_i \Big]^q.	
	\end{align*}
	Both terms can then be treated as in Theorem \ref{T:ekv-antigop}. A slight difference concerns the second one for which we just have to use the ``dual version'' of Lemma \ref{L:prop24} (see Remark \ref{R:dua2}) instead of the standard one.
\end{proof}	

\begin{remark}
	It goes without saying that Theorems \ref{T:ekv-antigop} and \ref{T:ekv-gop} may be restated in a~``dual form'' by replacing each symbol ``$\le$'' in their statements by ``$\ge$'' and vice versa.
\end{remark}

At this point we may apply the obtained results to establish an~interesting characterization of a~discrete inequality by a~continuous one in the case $p\in(0,1]$.

\begin{corollary}\label{T:main-prop2}
	Let $p\in(0,1]$ and $q\in(0,\infty)$. Let $\squ,\sqv,\sqw\in\RpZ$. Define $\squ,\sqv$ and $\sqw$ as in Theorem~\ref{T:main-prop1}. Then~\eqref{E:d-gop} holds for every sequence $\sqa\in\RpZ$ if and only if
	\begin{equation*}
	\left(\int_\R \left(\sup_{s\ge t}u(s)^p\int_{-\infty}^sf(y)\dy\right)^\qp w(t)\dt\right)^\pq
	\le \Cr{discrete_supremal}^p
	\int_\R f(t) v(t)\dt
	\end{equation*}
	holds for every $f\in\MM$.
	
	Similarly, \eqref{E:d-antigop} holds for every sequence $\sqa\in\RpZ$ if and only if
	\begin{equation*}
	\left(\int_\R \left(\sup_{s\ge t}u(s)^p\int_s^{\infty} f(y)\dy\right)^\qp w(t)\dt\right)^\pq
	\le \Cr{d-antigop}^p
	\int_\R f(t) v(t)\dt
	\end{equation*}
	holds for every $f\in\MM$.
\end{corollary}

\section{Proofs}

Let us start by proving Theorem~\ref{T:main-prop1} from the introduction.

\begin{proof}[Proof of Theorem~\ref{T:main-prop1}]
Suppose that~\eqref{E:d-gop} holds and let $f\in\MM$. Set $a_n=\int_n^{n+1} f$ for $n\in\Z$. Then we get, using the H\"older inequality,
	\begin{equation}\label{bla1}
	    \left(\sum_{n\in\Z}a_n^p v_n\right)^\jp
	    \le
	    \left(\sum_{n\in\Z}\int_n^{n+1} f(t)^p v_n\dt \right)\sp{\frac 1p}
	    =
	    \left(\int_{0}\sp{\infty}f(t)^p v(t)\dt\right)^\jp
	\end{equation}
and
\begin{align}
	\left(
    	\sum_{n\in\Z}\left(\sup_{i\ge n}u_i \sum_{k\le i} a_k\right)^q w_n
    \right)\sp{\frac 1q}
    &=\left(
    	\sum_{n\in\Z} \left(\sup_{i\ge n}u_i\int_{-\infty}^{i+1} f(y)\dy\right)^q \int_n^{n+1} w(t)\dt
    \right)\sp{\frac 1q} \notag \\
    &=\left(
    	\sum_{n\in\Z}\int_n^{n+1} \left(\sup_{s\ge t} u(s)\int_{-\infty}^ s f(y)\dy\right)^q w(t)\dt
        \right)\sp{\frac 1q} \notag \\
    &=\left(
	    \int_\R \left(\sup_{s\ge t} u(s)\int_{-\infty}^ s f(y)\dy\right)^q w(t)\dt
	    \right)\sp{\frac 1q}, \label{bla2}
\end{align}
and~\eqref{E:spoj-gop} follows.

Conversely, assume that~\eqref{E:spoj-gop} is satisfied. Let $\sqa\in\RpZ$ be arbitrary. Define
\begin{equation*}
    f=\sumZ a_n\chi_{[n,n+1)}.
\end{equation*}
Then we get \eqref{bla2} as above, and \eqref{bla1} holds now with identity in place of inequality. Hence, \eqref{E:d-gop} follows.

The equivalence between \eqref{E:d-antigop} and \eqref{E:spoj-antigop} can be obtained analogously.
\end{proof}

Now we can complete the proofs of the main results.

\begin{proof}[Proof of Theorem~\ref{T:main-discrete-p-big}]
	Let $\squ,\sqv,\sqw$ be as in Theorem~\ref{T:main-prop1}. Let $\Cr{discrete_supremal}$ be the least constant (including the possibility $\Cr{discrete_supremal}=\infty$) such that \eqref{E:d-gop} holds for all $\sqa\in\RpZ$.
	
	Assume that $1<p\le q$.
	From Theorem~\ref{T:main-prop1} and \cite[Theorem 4.1]{opic2006weighted} it follows that
	\begin{align*}
	 	\Cr{discrete_supremal}  & \approx \sup_{t\in\R} \sup_{x\ge t}u(x) \left(\int_{-\infty}^t w(s)\ds \right)^\jq \left(\int_{-\infty}^t v(s)^{\frac{1}{1-p}}\ds\right)^{\frac{p-1}{p}}\\
		 &	\qquad + \sup_{t\in\R} \left( \int_t^\infty \sup_{y\ge s}u(y)^{q} w(s)\ds\right)^{\frac{1}{q}} \left(\int_{-\infty}^t v(s)^{\frac{1}{1-p}}\ds\right)^{\frac{p-1}{p}} \\
		 & = B_1 + B_2.
	\end{align*}
	Notice that \cite[Theorem 4.1]{opic2006weighted} is stated for inequality \eqref{E:spoj-gop} in which the integration domain is replaced by $(0,\infty)$ and where the function $u$ is continuous. Therefore, to get the result in the form we need, we have to use a~change of variables and a~monotone approximation of $u$ by continuous functions.
	Anyway, we have
		\begin{align*}
			B_1 & = \sup_{n\in\Z} \sup_{t\in[n,n+1)} \sup_{x\ge t}u(x) \left(\int_{-\infty}^t w(s)\ds \right)^\jq \left(\int_{-\infty}^t v(s)^{\frac{1}{1-p}}\ds\right)^{\frac{p-1}{p}}\\
			 & = \sup_{n\in\Z} \uhdol_n \Bigg( \sum_{i\le n} w_i \Bigg)^\jq \Bigg( \sum_{k\le n} v^\frac1{1-p}_k \Bigg)^\frac{p-1}p
		\end{align*}
	and
		\begin{align*}
			B_2 & = \sup_{n\in\Z} \sup_{t\in[n,n+1)} \left( \int_t^\infty \sup_{y\ge s}u(y)^{q} w(s)\ds\right)^{\frac{1}{q}} \left(\int_{-\infty}^t v(s)^{\frac{1}{1-p}}\ds\right)^{\frac{p-1}{p}}\\
			& = \sup_{n\in\Z} \sup_{t\in[n,n+1)} \left( \int_t^{n+1} \sup_{y\ge s}u(y)^{q} w(s)\ds \int_{n+1}^\infty \sup_{y\ge s}u(y)^{q} w(s)\ds\right)^{\frac{1}{q}} \\
			& \hspace{80pt} \times \left( \int_n^t v(s)^{\frac{1}{1-p}}\ds + \int_{-\infty}^n v(s)^{\frac{1}{1-p}}\ds  \right)^{\frac{p-1}{p}}\\
			& = \sup_{n\in\Z} \sup_{\lambda\in[0,1)} \Bigg( \lambda\, \uhdol_n^q w_n + \sum_{i\ge n+1} \uhdol_i^q w_i \Bigg)^{\frac{1}{q}} \Bigg( (1-\lambda) v_n^{\frac{1}{1-p}} + \sum_{k\le n-1} v_k^{\frac{1}{1-p}} \Bigg)^{\frac{p-1}{p}}\\
			& \approx \sup_{n\in\Z} \Bigg( \sum_{i\ge n} \uhdol_i^q w_i \Bigg)^{\frac{1}{q}} \Bigg( \sum_{k\le n} v_k^{\frac{1}{1-p}} \Bigg)^{\frac{p-1}{p}}.
		\end{align*}
	To verify the latter equivalence, observe that
		$$
			\sup_{\lambda\in[0,1]} (X+\lambda x)^\alpha (Y + (1\!-\!\lambda) y)^\beta \ \le \ (X+x)^\alpha(Y+y)^\beta\ \le\ 2^{\alpha+\beta} \sup_{\lambda\in[0,1]} (X+\lambda x)^\alpha (Y + (1\!-\!\lambda) y)^\beta
		$$
	holds for all $x,y,X,Y\in[0,\infty)$ and $\alpha,\beta\in(0,\infty)$. (In case of doubts set $\lambda=\frac12$.) Combining the obtained estimates gives (i).
	
	If $1<p$ and $q>p$, then we use Theorem~\ref{T:main-prop1} and \cite[Theorem 4.4]{opic2006weighted} and proceed similarly as above.
	
	In the remaining cases where $p\le 1$ we use Corollary~\ref{T:main-prop2}, \cite[Theorems 4.1, 4.4]{opic2006weighted} and proceed analogously again.
\end{proof}

\begin{proof}[Proof of Theorem \ref{T:main-antigop}]
  This proof is analogous to that of Theorem~\ref{T:main-discrete-p-big}. We use Theorem~\ref{T:main-prop1}, Corollary~\ref{T:main-prop2} and the characterizations concerning inequalities for positive functions which are found in \cite[Theorems 6 and 7]{Kre3}. Details are omitted.
\end{proof}

\bibliography{bibfile}
\end{document}